\documentclass[12pt]{article}

\usepackage{microtype}
\usepackage[latin9]{inputenc}
\usepackage{amstext}
\usepackage{amsthm}
\usepackage{amssymb}
\usepackage{soul}
\usepackage{xfrac}
\usepackage{amsmath}
\usepackage{bm}
\usepackage{enumitem}
\usepackage[dvipsnames]{xcolor}

\makeatletter

\numberwithin{equation}{section}
\numberwithin{figure}{section}

\usepackage{amsthm}
\usepackage{fullpage}
\usepackage{makeidx}
\usepackage{amsfonts}
\usepackage{latexsym}
\numberwithin{equation}{section}

\numberwithin{figure}{section}

\newtheorem{thm}{Theorem}[section]
\newtheorem{prop}[thm]{Proposition}
\newtheorem{lemma}[thm]{Lemma}
\newtheorem{defn}[thm]{Definition}

\theoremstyle{remark}
\newtheorem{rem}[thm]{Remark}
\newtheorem{exa}[thm]{Example}
\usepackage{babel}

\newcommand{\Z}{\mathbb{Z}}

\newcommand{\C}{\mathbb{C}}

\newcommand{\q}{\quad}
\newcommand{\Aut}{\mathrm{Aut}}
\newcommand{\Rep}{\mathrm{Rep}}
\newcommand{\irr}{\mathrm{Irr}\,}

\usepackage{babel}

\usepackage[blocks]{authblk}
\title{  Permutation orbifolds and simple current extensions}
\date{}
\author{Ching Hung Lam and Nina Yu}

\makeatother

\begin{document}
\maketitle

\begin{abstract} In this article, we study permutation orbifolds and simple current extensions in the framework of vertex operator (super)algebras. We extend the construction of permutation-twisted modules for tensor products of vertex operator algebras to vertex operator superalgebras with $\frac12\mathbb Z$-grading, including the effect of the canonical involution. 
Using tensor category methods and simple current extensions, we build an induction theory for permutation-twisted modules associated with solvable automorphism groups, arising from semidirect products of simple current automorphisms and cyclic permutations.
In particular, we describe the structure and classification of irreducible twisted modules in terms of stabilizer subgroups and associated projective representations, and determine their multiplicities explicitly. As applications, we illustrate the theory with explicit examples from code vertex operator algebras, lattice-type simple current extensions, and the Moonshine vertex operator algebra.

\vspace{.5em}

\noindent{\bf Keywords}: Vertex superalgebras, twisted modules, simple current extensions, permutation orbifolds.
\vspace{.5em}
\end{abstract}

\section{Introduction}
Let $V$ be a vertex operator algebra (abbreviated as VOA). Let $G\le\Aut(V)$  be a finite automorphism subgroup. The fixed-point subalgebra
\[
V^G=\{ v\in V\mid gv=v \text{ for all } g\in G\}
\] 
is often called the  {\it orbifold subVOA} of $V$.  Orbifold theory mainly studies the representation theory of $V^G$ and its relation to that of $V$. 

When $V$ is $C_2$-cofinite and rational and $G$ is finite, there is a famous conjecture that $V^G$ is also $C_2$-cofinite and rational. Moreover, all irreducible $V^G$-modules can be realized as submodules of some irreducible $g$-twisted modules of $V$ with $g\in G$. This conjecture is proved for finite solvable groups \cite{CM,Mi}, but the conjecture is still open for a general finite group $G$.  Under the  assumption that  $V^G$ is $C_2$-cofinite and rational,  it is proved in \cite{DRX} that all irreducible $V^G$-modules can be realized as submodules of some irreducible $g$-twisted $V$-modules  with $g\in G$. 

Thus, the construction and classification of twisted modules play a central role in orbifold theory. In   \cite[Theorem 9.1]{DLM2}, it is also proved that if $V$ is a simple $C_2$-cofinite VOA and that $g\in\Aut(V)$ has finite order, then $V$ has at least one simple $g$-twisted
module.  Unfortunately, explicit constructions for $g$-twisted modules are not known in general.

As far as we know, such constructions are available only for a few special classes of VOAs and certain specific types of automorphisms. One fundamental construction applies to lattice VOA $V_L$ associated with a positive definite even lattice $L$ and a lift of an isometry $g\in O(L)$ \cite{Le,DL}. Another important class arises from inner automorphisms. More precisely, suppose that a VOA $V$ contains  a primary weight one vector $h$ such  that $h(0)$ acts semisimply on $V$ and with eigenvalues in $\frac{1}T \Z$  for some positive integer $T$. In this case,  there is an explicit construction for $\sigma_h$-twisted modules, where $\sigma_h= e^{2\pi i h(0)} \in \Aut(V)$ \cite{DLM}.

There are also constructions coming from extension theory.  Suppose that $V$ is a simple current extension of   a simple rational $C_2$-cofinite VOA $V^0$ of CFT type. In this setting, the representation theory of $V$  can be developed from that   of $V^0$  \cite{La,Y1}. In particular, a construction of some twisted modules using a notion of induced modules has been discussed (see Section \ref{S:3.4}). 
Another basic and highly important case is given by permutation orbifolds. Let $V$ be a VOA and let $k\ge 2$ be a positive integer. Consider the tensor product VOA
$V^{\otimes k}$. Then the $k$-cycle
$\sigma= (1 2\cdots k)\in S_k$ acts naturally as an automorphism of $V^{\otimes k}$ by permuting the tensor factors. In this setting,  explicit constructions of $\sigma$-twisted $V^{\otimes k}$-modules are given in \cite{BDM} (see Section \ref{sec3}). 

Despite the progress described above, a general construction theory for twisted modules remains incomplete.  One of the main aims of this article is to combine permutation automorphisms with simple current symmetries  and to study and analyze  the structures of the corresponding twisted modules systematically. 
The motivation for studying this problem is that these two constructions occur naturally in many important examples. 
Code VOAs, framed VOAs, and more general simple current extensions of tensor product VOAs carry both permutation symmetries of the tensor factors and simple current symmetries coming from the extension. There are also many such examples in the famous Moonshine VOA.  

Let $V$ be a  simple rational $C_2$-cofinite VOA of CFT type. Suppose that $U$ is a simple current extension of a tensor product VOA $V^{\otimes k}$. Let $g$ be an automorphism of $U$ such that $g$
acts as a permutation of the tensor factors of $V^{\otimes k}$.  In this case, every irreducible $g$-twisted module of $U$ is a direct sum of  irreducible $g$-twisted $V^{\otimes k}$-modules. The main task is to understand how an irreducible $g$-twisted module of $U$ will be decomposed as irreducible $g$-twisted $V^{\otimes k}$-modules. Therefore, one needs to understand how the extension data act on these sectors. This interaction produces genuinely new features, such as stabilizer subgroups, projective representations, and multiplicities. We believe that understanding these cases will be a natural step towards a more systematic construction for twisted modules.

In this article, under the above assumptions on $V$, we study simple current extensions of the tensor product VOA $V^{\otimes k}$. 
Using the framework of modular tensor categories and  the theory of simple current extensions, we develop an induction theory for permutation-twisted modules associated with solvable automorphism groups,  arising from semidirect products of simple current automorphisms and cyclic permutations.
In particular, we describe the structure and classification of irreducible twisted modules in terms of stabilizer subgroups and associated projective representations. We determine the decomposition rules and   multiplicities of induced modules under the extension. In addition, we consider the tensor product of vertex operator superalgebras with $\frac12\mathbb Z$-grading. We also discuss the construction of permutation-twisted modules for vertex operator subalgebras associated with some even binary codes and  analyze the effect of the canonical involution on the resulting module structure.

The organization of this article is as follows. In Section 2, we recall the basic notions of vertex operator superalgebras, their twisted modules, and  fix some necessary notation. In Section 3, we review the theory of permutation orbifolds. In particular,  we recall the construction of $\sigma$-twisted $V^{\otimes k}$-modules from \cite{BDM}, where $\sigma= (1 2\cdots k)\in S_k$ is a $k$-cycle acting naturally on $V^{\otimes k}$ by permuting the tensor factors. We also discuss a generalization of this construction to  tensor product of VOSA with $\frac12\mathbb Z$ grading \cite{B,BW}. In addition, we review the tensor products for twisted modules using the categorical approach developed in \cite{KO,CKM}. For later use, we also review results on fusion rules between untwisted $V^{\otimes k}$-modules and $\sigma$-twisted $V^{\otimes k}$-modules obtained in \cite{DLXY}.
 In Section 4, we recall the notion of simple current extension and discuss their representation theory. In particular, a notion of induced modules is discussed.  
 Section 5 contains our main results. We study simple current extensions of the tensor product VOA $V^{\otimes k}$ and develop an induction theory for permutation-twisted modules associated with some solvable automorphism groups  arising as semidirect products of simple current automorphisms and cyclic permutations. We also describe the structure and classification of irreducible twisted modules in terms of stabilizer subgroups and associated projective representations. 
Finally, in Section 6, we illustrate the theory with several explicit examples. 

\section{Basics}\label{sec:2}

In this section, we review some basic notions of vertex  operator (super)algebra and their twisted modules
\cite{Bo,DL,FHL,FLM} and fix some necessary notation. Throughout this
paper, $z_{0}, z_{1},z_{2}$ are independent commuting formal variables.

\emph{A super vector space} is a $\mathbb{Z}_{2}$-graded vector space
$V=V_{\overline{0}}\oplus V_{\overline{1}}.$ The elements in $V_{\overline{0}}$ (resp.
$V_{\overline{1}}$) are called even (resp. odd). Let $\left|v\right|$
be 0 if $v\in V_{\overline{0}}$ and $1$ if $v\in V_{\overline{1}}.$

\begin{defn}A \emph{vertex superalgebra} is a quadruple $(V, {\bf 1},D,Y)$, where $V=V_{\overline{0}}\oplus V_{\overline{1}}$
is a $\mathbb{Z}_{2}$-graded vector space, $D$ is an endomorphism
of $V$, $1\in V_{\overline{0}}$ is  the vacuum vector of $V$, and
$Y$ is a linear map
\begin{alignat*}{1}
Y\left(\cdot,z\right):  & V\to\left(\mathrm{End}V\right)\left[\left[z,z^{-1}\right]\right]\\
 & v\mapsto Y\left(v,z\right)=\sum_{n\in\mathbb{Z}}v_{n}z^{-n-1}\ (\text{where\ }v_{n}\in\mathrm{End}V)
\end{alignat*}
 such that the following axioms hold: 
\begin{enumerate}[label=\textnormal{(\arabic*)}]

	\item  (Lower truncation)  For any $u, v\in V, u_{n}v=0$ for all sufficiently large $n$.

\item (Translation covariance) $\left[D, Y\left(v, z\right)\right]=Y\left(D\left(v\right), z\right)=\frac{d}{dz}Y\left(v, z\right)$
for all $v\in V.$

\item (Vacuum) $Y\left({\bf 1}, z\right)=\mathrm{Id}_{V}$ (the identity operator of
$V$).
	\item (Creation) $Y\left(v, z\right){\bf1}\in\left(\mathrm{End}V\right)\left[\left[z\right]\right]$
and $\lim_{z\to0}Y\left(v,z\right){\bf 1}=v$ for all $v\in V.$

\item (Super Jacobi identity) For $\mathbb{Z}_{2}$-homogeneous elements $u, v\in V$, the following
Jacobi identity holds:
\begin{gather*}
z_{0}^{-1}\delta\left(\frac{z_{1}-z_{2}}{z_{0}}\right)Y\left(u,z_{1}\right)Y\left(v,z_{2}\right)-\left(-1\right)^{\left|u\right|\left|v\right|}z_{0}^{-1}\delta\left(\frac{z_{2}-z_{1}}{-z_{0}}\right)Y\left(v,z_{2}\right)Y\left(u,z_{1}\right)\\
=z_{2}^{-1}\delta\left(\frac{z_{1}-z_{0}}{z_{2}}\right)Y\left(Y\left(u,z_{0}\right)v,z_{2}\right).
\end{gather*}
\end{enumerate}
\end{defn}

\begin{defn}
    
A \emph{ vertex operator superalgebra (VOSA)} is a vertex superalgebra $(V, {\bf 1},D,Y)$ together with a conformal vector $\omega\in V_{\overline 0}$ such that 
$$Y\left(\omega,z\right)=\sum_{n\in\mathbb{Z}}L\left(n\right)z^{-n-2}$$ and the operators $L(n)$ satisfy:
\begin{enumerate}[label=\textnormal{(\arabic*)}]
\item $\left[L\left(m\right),L\left(n\right)\right]=\left(m-n\right)L\left(m+n\right)+\frac{1}{12}\left(m^{3}-m\right)\delta_{m+n,0}c,$
	for all $m, n\in\mathbb Z$ and for some $c\in\mathbb C$, called the \emph{central charge}. 
	
\item  $L\left(-1\right)=D,$ i.e., $\frac{d}{dz}Y\left(v,z\right)=Y\left(L\left(-1\right)v,z\right)$
	for $v\in V.$
	
\item $L(0)$ acts semisimply on $V$ with eigenvalues in $\frac{1}{2}\mathbb{Z}.$
	
\item  $\dim V_k<\infty $ and $V_k=0$ for sufficiently small $k$, where $V_k= \{ v\in V\mid L(0) v=kv\}$ for $k\in \frac{1}{2}\mathbb{Z}$.  If $v\in V_k$, $k$ is called the \emph{weight} (or conformal weight) of $v$ and denoted by $\mathrm{wt}v$.
\end{enumerate}
	
\end{defn}

\noindent \textbf{Remark.} A vertex algebra (VA) is the special case of a vertex superalgebra with trivial odd part $V_{\bar 1}=0.$  A VOA is a conformal vertex algebra.

\begin{defn}   An
\emph{automorphism} $g$ of
{
 a vertex algebra $V$ is a linear automorphism of $V$
 such that $gY\left(v,z\right)g^{-1}=Y\left(gv,z\right)$
for all $v\in V$ (and hence $g\bf{1}=\bf{1}$).

If $V$ is a VOSA with the conformal vector $\omega$, we further require that $g\omega =\omega$, that is,   $g$  preserves the conformal structure. 
}
\end{defn}

 Let $V$ be a vertex superalgebra and let $g\in \Aut(V)$ be
of finite order $T$. Fix a primitive $T$-th root of unity $\eta$. Then $V$ decomposes as a  direct sum of the
eigenspaces $V^{j}$ of $g$:
\[
V=\bigoplus_{j\in\mathbb{Z}/T\mathbb{Z}}V^{j},
\]
where $V^{j}=\left\{ v\in V\mid gv=\eta^{j}v\right\}.$

\begin{defn}  Let $V$ be a vertex superalgebra. A \emph{weak g-twisted $V$-module
}is a vector space $M$
equipped
with a linear map
\begin{align*}
Y_{g}(\cdot,z):\  & V\to({\rm End}\;M)[[z^{1/T},z^{-1/T}]]\\
 & v\mapsto Y_{g}(v,z)=\sum_{n\in\frac{1}{T}\mathbb{Z}}v_{n}^{g}z^{-n-1},
\end{align*}
  such
that for $u,v\in V$ and $w\in M$ the following conditions hold:
\begin{enumerate}[label=\textnormal{(\arabic*)}]
\item For any $v\in V$ and $ w\in M$, $v_{n}^{g}w=0$ for $n$ sufficiently
large.

\item $Y_{g}({\bf 1},z)=\mathrm{Id}_{M}$ (the identity operator on $M$).

\item For $u, v\in V$ of homogeneous parity, the following Jacobi identity
holds
\begin{gather*}
z_{0}^{-1}\delta\left(\frac{z_{1}-z_{2}}{z_{0}}\right)Y_{g}(u,z_{1})Y_{g}(v,z_{2})-(-1)^{|u||v|}z_{0}^{-1}\delta\left(\frac{z_{2}-z_{1}}{-z_{0}}\right)Y_{g}(v,z_{2})Y_{g}(u,z_{1})\\
=\frac{z_{2}^{-1}}{T}\sum_{j\in\mathbb{Z}/T\mathbb{Z}}\delta\left(\eta^{j}\frac{\left(z_{1}-z_{0}\right)^{1/T}}{z_{2}^{1/T}}\right)Y_{g}(Y(g^{j}u,z_{0})v,z_{2}).
\end{gather*}
\end{enumerate}
If $g=1$, then a weak $g$-twisted $V$-module is a weak $V$-module.
\end{defn}

Throughout this paper, $ \mathbb N=\{0,1,2,\ldots\}$ denotes the set of all natural numbers. 

\begin{defn} 
Let $V$ be a VOSA. A \emph{$\mathbb N$-graded $g$-twisted $V$-module} is a weak $g$-twisted $V$-module $M$ which carries a $\mathbb N$-grading (or a $\frac{1}{2T}\mathbb{N}$-grading), i.e., $$M=\bigoplus_{k\in \mathbb{ N}}M(\frac{1}{2T} k)=\bigoplus_{n\in \frac{1}{2T} \mathbb{ N}} M(n)$$ such that, for homogeneous $v\in V$, $m\in \frac{1}T \Z, n\in \frac{1}{2T} \mathbb{N}$,  $$v^g_mM(n)\subseteq M(n+\mathrm{wt}v-m-1).$$ 
We may assume that $M(0)\neq 0$.  Note also that $\mathrm{wt}v\in \frac{1}2 \Z$. 

If $g=1$, then $M$ is called a $\mathbb N$-graded $V$-module.
\end{defn}

\begin{defn}
		Let $V$ be a VOSA.	
An ordinary \emph{$g$-twisted $V$-module} is
a weak $g$-twisted $V$-module $M$ satisfying the condition that
$M=\oplus_{h\in\mathbb{C}}M_{(h)}$, where $M_{(h)}=\left\{ w\in M\mid L(0)^{g}w=hw\right\}$ for $L\left(0\right)^{g}=\omega_{1}^{g},$
$\dim M_{(h)}<\infty$ for all $h\in\mathbb{C}$, and $M_{(h+\frac{n}{T})}=0$
for fixed $h\in\mathbb{C}$ and for all sufficiently small integers
$n$. 
\smallskip

If $g=1$, then an ordinary $g$-twisted $V$-module is an ordinary
$V$-module.

\end{defn}

\medskip

Now we consider the action of Aut$\left(V\right)$ on twisted modules.
Let $g,h\in\text{Aut}\left(V\right)$ with $g$ of finite order.
If $\left(M,Y_{M}\right)$
is a weak $g$-twisted $V$-module, there is a weak $h^{-1}gh$-twisted
$V$-module $\left( M\circ h,Y_{ M\circ h}\right)$, where $ M\circ h\cong M$
as a vector space and the vertex operator map is given by  $Y_{ M\circ h}\left(v,z\right)=Y_{M}\left(hv,z\right)$
for $v\in V$. This defines a right action of $\text{Aut}\left(V\right)$
on weak twisted $V$-modules (and hence on the isomorphism classes of weak twisted
$V$-modules). We write $$ \left(M,Y_{M}\right)\circ h:=\left( M\circ h,Y_{M\circ h}\right)=M\circ h.$$
We say $M$ is \emph{$h$-stable} if $M\cong M\circ h$ as twisted $V$-modules.
In particular, if $M$ is an $\mathbb N$-graded $g$-twisted $V$-module, then   $ M\circ g\cong M$  \cite{DLM2}.

\section{Permutation orbifold}\label{sec3}
Let $V$ be a VOA and let $k\ge 2$ be a positive integer. 
Let $\sigma= (1 2\cdots k)\in S_k$ be the $k$-cycle, which acts naturally as an automorphism of $V^{\otimes k}$ by permuting the tensor factors. 
In this section, we first recall the construction of $\sigma$-twisted modules for $V^{\otimes k}$ arising from permutation orbifolds, together with its extension to vertex operator superalgebras. We also review tensor product structures and the correspondence between intertwining operators and fusion rules, which will be used later to study tensor product decompositions. 
\subsection{Structure of twisted modules}
We first recall the construction of $\sigma$-twisted $V^{\otimes k}$-modules from \cite{BDM}. Let 
\[
\Delta_k(z)=\exp\left(\sum_{n\geq 1}a_nz^{-\frac{n}k}L(n)\right)k^{-L(0)}z^{(1/k-1)L(0)}, 
\]
where the coefficients $a_n (n \geq 1)$ are uniquely determined by the identity
\[
\exp \left( 
\sum_{n\geq 1} -a_n x^{n+1}\frac{d}{dx} \right) x= \frac{1}k(1+x)^{k} -\frac{1}k.
\]

The following is one of the main results of \cite{BDM}. 
\begin{thm}\label{thm:3.1} Let $V$ be a VOA and let  $\sigma= (1 2\cdots k) \in\Aut(V^{\otimes k})$. Then for any (weak,
$\mathbb N$-graded) $V$-module $(W, Y_W )$, there exists a (weak, $\mathbb N$-graded) $\sigma$ twisted $V^{\otimes k}$-module $(T_{\sigma}(W), Y_{T_{\sigma}(W )} )$ such that
$T_{\sigma}(W)=W$ as a vector space and 
$$Y_{T_{\sigma}(W)}(u^1, z)=Y_W(\Delta_k(z)u, z^{1/k}), \quad u \in V,$$  where 
$u^1=u\otimes \textbf{1}^{\otimes(k-1)}\in V^{\otimes k}$. Furthermore, every (weak, $\mathbb N$-graded) $\sigma$-twisted $V^{\otimes k}$-module is isomorphic to one of this form.
\end{thm}

The construction above can also be generalized to VOSA with $\frac{1}{2}\Z$-grading \cite{B,BW,MS}. 

Let  $V=V_{\overline{0}}\oplus V_{\overline{1}}$ be a VOSA
such that $V_{\overline{0}} =\oplus_{n\in \Z} V_n$ and $V_{\overline{1}}= \oplus_{n\in \frac{1}2+\Z} V_n $.  
    For $\alpha=(\overline{\alpha_1}, \dots, \overline{\alpha_k})\in \mathbb Z_2^k$, define $$V^\alpha:=V_{\overline{\alpha_1}}\otimes \cdots \otimes V_{\overline{\alpha_k}}\subset V^{\otimes k}.$$ 
Then $$V^{\otimes k}=\oplus_{\alpha\in\mathbb Z_2^k}V^\alpha,$$ and we may view $V^\alpha$ as a subspace of $V^{\otimes k}$. 
Recall  that
$V^{\otimes k}$ is again a VOSA, with the vertex operator  determined by
\[
Y_{V^{\otimes k}} (u_1\otimes \cdots \otimes u_k,z ) (
v_1\otimes \cdots \otimes v_k)
= (-1)^{\langle \alpha, \beta\rangle} Y_V(u_1,z)v_1\otimes \cdots
\otimes Y_V(u_k,z)v_k,
\] 
for $u_1\otimes\cdots\otimes u_k\in V^\alpha$ 
and $v_1\otimes\cdots\otimes v_k\in V^\beta$, where  
\(
\langle \alpha,\beta\rangle 
:= \sum_{1\le j<i\le k} \alpha_i\,\beta_j  \in \mathbb{Z}_2.
\)
Define $\Delta_k(z)$ as in the VOA case, 
\[
\Delta_k(z)=\exp\left(\sum_{n\geq 1}a_nz^{-\frac{n}k}L(n)\right)k^{-L(0)}z^{(1/k-1)L(0)}.
\]
Then for homogeneous $u\in V_{\overline{0}}\cup V_{\overline{1}}$, one has  $$\Delta_k(z)u \in V[[z^{1/k}, z^{-1/k}]]z^{\frac{|u|(1-k)}{2k}},$$ where $|u|$ denotes the parity of $u$. 

As  a generalization of Theorem \ref{thm:3.1}, we have the following result. 

\begin{thm}[\cite{B,BW,MS}]\label{thmSVOSA}  Let $V$ be any VOSA with $\frac 1 2\Z$-grading. Let $\sigma= (1\  2\cdots k) \in\Aut(V^{\otimes k})$ and  $\theta$ be the canonical involution of $V$, acting as  $1$ on $V_{\overline{0}}$ and  as $-1$ on $V_{\overline{1}}$.  For $r=0, 1$, let $(W,Y_W)$ be a $\theta^r$-twisted $V$-module. Then $(T_{\sigma}(W), Y_{T_{\sigma}(W )} )$
is a $\theta_1^{r+k-1}\sigma$-twisted $V^{\otimes k}$-module, where
$T_{\sigma}(W)=W$ as a vector space and the vertex operator map $Y_{T_{\sigma}(W)}(\cdot, z)$ is uniquely determined by
$$Y_{T_{\sigma}(W)}(u^1, z)=Y_W\left(\Delta_k(z)u, z^{1/k}\right),\quad  u \in V,$$ with
$u^1=u\otimes 1^{\otimes(k-1)}\in V^{\otimes k}$ and $\theta_1=\theta\otimes 1\otimes \cdots\otimes 1\in \Aut(V^{\otimes k}).$ Furthermore, every  $\theta_1^{r+k-1}\sigma$-twisted $V^{\otimes k}$-module is isomorphic to one
of this form.
\end{thm}

\begin{rem}
In Theorem \ref{thmSVOSA}, the $\theta^r$-twisted $V$-module $W$ is assumed to be a ``parity-stable" twisted module. That means $W=W^{(0)}\oplus W^{(1)}$ is a $\Z_2$-graded vector space which is compatible with the $\Z_2$-grading of $V$. In particular, $u_n w \in W^{(i+j)}$ for any  $u\in V^{(i)}, w\in W^{(j)}$ and $n\in \Z$
\cite[Theorem 6.3]{BW}. 
\end{rem}

\subsection{Tensor products for twisted modules}

In this subsection, we recall the tensor category framework for studying tensor products of twisted modules. In particular, we review  the category $\Rep(V)$ and the induction and restriction functors relating $V^G$-modules and twisted $V$-modules.

Throughout this subsection, we assume that $V$ is a $C_2$-cofinite, rational and self-dual VOA of CFT type, and that $G$ is a finite solvable subgroup of $\Aut(V)$. Then the fixed point subalgebra $V^G$ is also  rational and $C_2$-cofinite \cite{CM,Mi}.

Denote by $\mathcal{C}_V$ the category of ordinary $V$-modules, and by $\mathcal{C}_{V^G}$ the category of ordinary $V^G$-modules. 
Both $\mathcal{C}_V$ and $\mathcal{C}_{V^G}$ are modular tensor categories \cite{H}. Furthermore, from \cite{KO}  and \cite{CKM}, the $V^G$-module $V$ is a commutative associative
algebra in $\mathcal{C}_{V^G}$. Recall that $V$ decomposes as a $ \C[G]\otimes  V^G$-module as $$V =\oplus_{\chi\in \irr(G)} M_\chi\otimes  V_\chi,$$ where $\irr(G)$ is the set of irreducible characters of $G$, $M_\chi$ is the irreducible $G$-module affording the character $\chi$, and $V_\chi$ is an irreducible
$V^G$-module \cite{DRX}.

\begin{defn} [\cite{KO,CKM}] 
Denote by $\Rep(V )$ the subcategory of $\mathcal{C}_{V^G}$ consisting of every $V^G$-module $W$ together with a $V^G$-intertwining operator $Y_W (\cdot , z)$ of type $\binom{W}{V W}$, satisfying: 

\begin{enumerate}[label=\textnormal{(\arabic*)}]
\item (Associativity) For any $u,v \in V, w\in W$ and $w' \in W'$, the formal series
$$\langle w', Y_W (u, z_1)Y_W (v, z_2)w\rangle, \quad \text{and}\quad \langle w',Y_W(Y(u,z_1-z_2)v,z_2)w\rangle$$
converge on the domains $|z_1|>|z_2|>0$ and $|z_2|>|z_1-z_2|>0$, respectively, to multivalued analytic functions which coincide on the intersection of their domains. 

\item (Unit) $Y_W(1, z)=Id_W$.
\end{enumerate}
\end{defn}
In \cite{KO} (see also \cite{CKM}), a categorical tensor product functor $\boxtimes_V$ in the category of $\Rep(V )$ has been discussed. It is associative,  and for any two $V^G$-modules
$M, N$ in $\Rep(V )$,  the object $M \boxtimes_V N $ is a quotient of $M\boxtimes_{V^G} N$. Moreover, $\Rep(V )$ is a fusion category. In particular, $\Rep(V )$ is a semisimple category with finitely many inequivalent simple objects. 

\begin{thm}[{\cite[Theorem 1.6]{KO}}]\label{KO}
Define  functors $F:\mathcal{C}_{V^G} \to \Rep(V)$ and $R:\Rep(V)\to \mathcal{C}_{V^G}$ by
$$F(M)= V\boxtimes_{V^G} M\quad \text{and}\quad R(W)= W|_{V^G}.$$ Then
\begin{enumerate}[label=\textnormal{(\arabic*)}]
\item  Both $F$ and $R$ are exact and injective on morphisms.
\item $F$ and $R$ are adjoint: one has canonical functorial isomorphisms
\[
\mathrm{Hom}_V(F(M),X) = \mathrm{Hom}_{\mathcal{C}_{V^G}} (M,R(X)),  \quad M\in \mathcal{C}_{V^G}, X\in \Rep(V).
\]

\item $F$ is a tensor functor: one has canonical isomorphisms 
\[
 F( M\boxtimes_{V^G} N)= F(M)\boxtimes_V F(N), \quad F(V^G)= V.
 \]
\item There are canonical isomorphisms:  
\[
R(F(M)) = V\boxtimes_{V^G} M, \quad \text{ and } \quad R (F(M)\boxtimes_V X) = M\boxtimes_{V^G}R(X).
\]
\end{enumerate}
\end{thm}

The following lemmas can be found in \cite{DLXY}(see also \cite{K1, K2}). 
\begin{lemma}
Let $g\in G$ and $W$ be a $g$-twisted $V$-module. Then $W$ is an object of $\Rep(V )$. Furthermore, if $W_i$ is a $g_i$-twisted $V$-module with $g_i\in G$ for $i = 1, 2$, then $W_1$ and $W_2$ are isomorphic as  objects in $\Rep(V )$ if and only if
$g_1=g_2$ and $W_1=W_2$ as $g_1$-twisted $V$-modules.
\end{lemma}

\begin{lemma}
If $W$ is a simple object of $\Rep(V )$, then $W$ is an irreducible $g$-twisted $V$-module for some $g\in G$.
\end{lemma}

\subsection{Fusion rules}

The fusion rules between  untwisted $V^{\otimes k}$-modules and  $\sigma$-twisted modules of $V^{\otimes k}$ are studied in \cite{DLXY}. In this subsection, we recall the key correspondence between intertwining operators.

\begin{prop} Let $M, N, W$ be $V$-modules. 
	For any intertwining operator $\mathcal{Y}(\cdot,z)$ of type $\binom{W}{M\q N}$, there exists an intertwining operator $\overline{\mathcal{Y}}(\cdot,z)$ of
	type  $\binom{T_{\sigma}(W)}{M^1\q  T_{\sigma}(N)}$,  uniquely determined by
	\[
	\overline{\mathcal{Y}}(w^1,z)=\mathcal{Y}(\Delta_k(z) w, z^{1/k}), \quad w\in M.
	\]
    where $M^1=M\otimes V^{\otimes(k-1)}$ and $w^1=w\otimes\textbf{1}^{(k-1)}\in M^1$. 
\end{prop}

\begin{thm}
	Let $M,N,W$ be $V$-modules and let $\sigma= (1 2\cdots k) \in\Aut(V^{\otimes k})$. 
	Set $M^1 = M\otimes V^{\otimes (k-1)}$. Given  intertwining operator $\overline{\mathcal{Y}}(\cdot,z)$ of type $\binom{T_{\sigma}(W)}{M^1\q  T_{\sigma}(N)}$, define a linear map $\mathcal{Y}(\cdot,z) $ by  
	\[
	\mathcal{Y}(w,z) = \overline{\mathcal{Y}}\left( (\Delta_k(z^k)^{-1} w)^1, z^k\right), w\in M.
	\]
	Then $\mathcal{Y}(\cdot,z)$ is an intertwining operator of type $\binom{W}{M\q N}$.
\end{thm}

In particular, the assignment
$$	\pi: I_V \binom{W}{M\q N} \to I_{V^{\otimes k}} \binom{T_{\sigma}(W)}{M^1\q  T_{\sigma}(N)}, \quad 
	\mathcal{Y} (\cdot, z) \mapsto \overline{\mathcal{Y}}(\cdot,z),$$
is a linear isomorphism.

\begin{thm}
	Let $M$ and $N$ be $V$-modules. Suppose that the tensor product $(M\otimes_V N, \mathcal{Y})$ exists. 
	Then the tensor product $\left((M\otimes V^{\otimes (k-1)})\boxtimes_{V^{\otimes k}}T_\sigma(N), \overline{\mathcal Y}\right)$ exists and
	$$(M\otimes V^{\otimes (k-1)}) \boxtimes T_{\sigma}(N) \cong T_{\sigma} (M\boxtimes_V N).$$ In particular, 
	$\left(M\otimes V^{\otimes (k-1)}\right) \boxtimes T_{\sigma}(V) \cong T_{\sigma} (M)$.
\end{thm}

In addition, the following theorem is proved in \cite{DLXY}.
\begin{thm}\label{FP}
	Let $V$ be a $C_2$-cofinite, rational and self-dual VOA of CFT type, and let $M_1, \dots, M_k, N $ be $V$-modules. Let $\sigma$ be a $k$-cycle. Then
	\[
	(M_1\otimes \cdots \otimes M_k)\boxtimes_{V^{\otimes k}}T_{\sigma} (N) 
	\cong T_{\sigma} (M_1\boxtimes_V \cdots \boxtimes_V M_k \boxtimes_V N).
	\]
\end{thm}

\section{Simple current extensions}
In this section, we recall some basic results on simple current extensions and their representation theory, following \cite{Y1}. We review structural properties of such extensions and describe how modules over a VOA extend to (twisted) modules over its simple current extensions, which will be used in later sections.

We begin by recalling the notion of simple current extensions from \cite{Y1}.
Let $V^0$ be a simple rational $C_2$-cofinite VOA of CFT type and 
let $\{ V^\alpha \mid \alpha \in D\}$ be a set of inequivalent irreducible 
$V^0$-modules indexed by an abelian group $D$.
A {\it $D$-graded extension} of $V^0$ is a simple VOA 
$$V_D=\oplus_{\alpha\in D} V^\alpha$$ such that $V^0$ is a full subVOA of $V_D$  and $V_D$ carries a $D$-grading via its vertex operator map
$$Y(v^\alpha,z)v^\beta\in V^{\alpha+\beta},\ \text{for}\  \text{all}\  v^\alpha \in V^\alpha, 
v^\beta \in V^\beta.$$

Recall that irreducible $V^0$-module $J$ is called a \emph{simple current $V^0$-module} if, for every irreducible $V^0$-module $M$, the fusion product $
J\boxtimes_{V^0} M
$
is again irreducible. Equivalently, fusion with $J$ defines a permutation of the set of isomorphism classes of irreducible $V^0$-modules.

If, in addition, each  $V^\alpha (\alpha \in D)$ is a  simple current $V^0$-module, then 
$V_D$ is called a {\it $D$-graded simple current extension} of $V^0$.
The rationality of $V^0$ implies $D$ is automatically finite. 

Next we list some basic  facts about simple current extensions.

\begin{prop}\label{prop:2.3}
  (\cite{ABD,DM,La,Y1})
  Let $V^0$ be a simple rational $C_2$-cofinite VOA of CFT type and let $V_D=\oplus_{\alpha\in D} V^\alpha$ be a $D$-graded simple current extension
  of $V^0$.
  Then
 \begin{enumerate}[label=\textnormal{(\arabic*)}]
\item $V_D$ is rational and $C_2$-cofinite.
 \item  If $\widetilde{V}_D=\oplus_{\alpha\in D}\widetilde{V}^\alpha$ is another $D$-graded 
  simple current extension of $V^0$ such that $\widetilde{V}^\alpha\simeq V^\alpha$ as 
  $V^0$-modules, then $V_D$ and $\widetilde{V}_D$ are isomorphic VOAs over $\C$.
  Such an isomorphism between $V_D$ and $\widetilde{V}_D$ may be chosen to extend a $V^0$-module isomorphism.

\item  For any subgroup $E$ of $D$, the subalgebra $V_E:=\oplus_{\alpha \in E} V^\alpha$ 
  is an $E$-graded simple current extension of $V^0$.
  Moreover, viewing $V_D$ as an extension of $V_E$, the VOA $V_D$ is a $D/E$-graded simple current extension of $V_E$.
  \end{enumerate}
\end{prop}

\begin{rem}\label{DandG}
	Let $V_D=\oplus_{\alpha\in D} V^\alpha$ be a $D$-graded simple current extension
	of $V^0$ and let $D^*= \irr (D)$ be the set of all characters of $D$. Then $D^*$ acts on $V_D$ by the action $$\chi \cdot u= \chi(\alpha) u, \quad \chi\in D^*, u\in V^\alpha, \alpha \in D.$$ In particular, $V^0= (V_D)^{D^*}$.    

Conversely, let $G\le\Aut(V)$ be a finite abelian subgroup and assume that $V$ is  a simple, rational, $C_2$-cofinite VOA of CFT type. Then  $V^G$ is also a simple, rational, $C_2$-cofinite VOA of CFT type. For $\chi\in \irr(G)$, set  $$V^\chi=\{ v\in V \mid gv=\chi(g) v \text{ for all } g\in G\}.$$ 
Note that all irreducible characters of $G$ are linear characters since $G$ is abelian.  
Therefore, we have  a decomposition $$V=\oplus_{\chi\in \irr (G)} V^\chi$$  where each $V^\chi$ is a  simple current $V^G$-module. In particular, $V$ is a  simple current extension of $V^G$ graded by $\irr( G)$.  
\end{rem}

\begin{prop}[\cite{Sh04}] \label{lift} 
Let $V_D=\oplus_{\alpha\in D} V^\alpha$ be a $D$-graded simple current extension
of $V^0$. Let $\sigma\in \Aut(V^0)$ be such that $$\{(V^\alpha)\circ \sigma\mid \alpha\in D \} = \{V^\alpha\mid \alpha\in D \}.$$ Then $\sigma$ admits a lift $\widetilde{\sigma}
\in \Aut(V_D)$,  i.e., $\widetilde{\sigma} (V^0) = V^0$ and $\widetilde{\sigma}|_{V^0} = \sigma$. Moreover, such a lift is unique up to multiples
of $D^*$.
\end{prop}

The representation theory of  simple current extensions is well developed in \cite{La,Y1}.
It is proved that any $V^0$-module can be extended to a certain twisted 
module over $V_D$.
We now review a result from \cite{La} that will be used later. 
	
	%
		
	%

\begin{thm} \label{decompose}
Let $V$ be  
a simple rational $C_2$-cofinite VOA of CFT type,  and let  $G\le \mathrm{Aut\,}(V)$ be a finite
	abelian subgroup.  
Let $M$ be an irreducible $V$-module and $\left\{
L_{1},\cdots ,L_{n}\right\} $ be the set of all inequivalent irreducible $V^{G}$-submodules of $M$. Then there exists a subgroup $H\le G$ and
irreducible projective representations $Q_{1},\cdots ,Q_{n}$ of the character group  $\left( G/H\right) ^{*}=\irr \left(
G/H\right) $  such
that, as $V^G$-modules, 
\begin{equation*}
M\cong\bigoplus _{i}Q_{i}\otimes L_{i},
\end{equation*}
and $\dim Q_1=\dim Q_2 =\cdots=\dim Q_n$.  
Moreover, for each  $i$, the tensor product $Q_{i}\otimes L_{i} $ is an irreducible $V^{H}$-module.
\end{thm}

\subsection{Induced modules} \label{S:3.4}

In this subsection, we recall the construction of induced modules for simple current extensions and describe their decompositions as modules over $V^0$ and over suitable intermediate extensions \cite{Y1}. 
Let $M$ be an irreducible $V_D$-module. Since $V^0$ is rational, 
$M$ contains an irreducible $V^0$-submodule $W$. Define the stabilizer subgroup 
 $$D_W:=\{\alpha \in D \mid V^\alpha \boxtimes_{V^0} W\simeq_{V^0} W\}.$$
Then $D_W$ is  a subgroup of $D$. 

Let $G=\irr (D)$. Then $G$ acts on $V_D$ and $ V^0=(V_D)^G$. 
Set $$H :=\{ \chi\in \irr (D)\mid  \chi(\alpha)=1 \text{ for all } \alpha\in D_W\}.$$  
Then $H$ consists of all characters of $D$ that are trivial on $D_W$. It follows that $V^H= V_{D_W}$, and $(G/H)^*\cong D_W$ as abelian groups.

Set $S_W= D/D_W$ and for any $s=\alpha+D_W \in S_W$,  define $$W^s= V^\alpha \boxtimes_{V^0} W.$$ Note that the isomorphism class of $W^s$ is independent of the choice of the representative $\alpha$ of the coset $s$ and $D_W=D_{W^s}$ by the associative and commutative properties of the fusion products.  

For each $\alpha\in D_W$ and $s\in S_W$, fix a $V^0$-intertwining operator $I^\alpha_s(\cdot, z)$ of type $\binom{W^s}{V^\alpha \ W^s}$. We normalize the intertwining
operator such that $I^0_s(1, z)=\mathrm{Id}_{W^s}$; in other words, $I^0_s(\cdot, z)$ is the vertex operator on the $V^0$-module $W^s$.   By
the associative property of intertwining operators, there are (nonzero)
scalars $\lambda_s(\alpha, \beta)\in \C$ such that the following equality holds:
\[
\langle \nu, I^\alpha_s(x^\alpha, z_1)I^\beta_s(x^\beta,z_2)w\rangle
= \lambda_s(\alpha, \beta) \langle \nu, I_s^{\alpha+\beta}( Y_{V_D}(x^\alpha, z_0)x^\beta,z_2)w\rangle|_{z_0=z_1-z_2}, 
\] 
where $x^\alpha\in V^\alpha, x^\beta\in V^\beta, w\in W^s$ and $\nu \in (W^{s})^*$.  By iterating the associative property, one can deduce that 
\[
\lambda_s(\alpha, \beta)\lambda_s(\alpha+\beta, \gamma) = \lambda_s(\alpha, \beta+\gamma)\lambda_s(\beta, \gamma),\quad \alpha, \beta, \gamma\in D_W.
\]
In other words, $\lambda_s$ defines a $2$-cocycle on $D_W$.
Let 
\[
1 \to \C^* \to \widehat{D_W} \to D_W \to 1 
\]
be a central extension of $D_W$ associated with  the  2-cocycle $\lambda_s$. Equivalently, $\lambda_s$ determines a twisted group algebra $\mathbb C^{\lambda_s}[D_W]$.
By Theorem \ref{decompose}, we have 
$$
M\cong \bigoplus_{s\in S_W} Q^s\otimes W^s,
$$
where $Q^s$ is an irreducible $\C^{\lambda_s}[D_W]$-module for each $s\in S_W$.

Let $W$ be an irreducible $V^0$-module. For $\alpha\in D$, define  
$$\chi_W(\alpha)= \rho(V^\alpha \boxtimes W) - \rho(W) \quad \mod \Z,
$$
where $\rho(W)$ is the conformal weight of the irreducible $V^0$-module $W$.
By the associativity and the fact that the powers of $z$ in an intertwining operator of type $I_{V^0}\binom{V^\alpha \boxtimes W}{V^\alpha \qquad W}$ are contained in $\rho(V^\alpha \boxtimes W) - \rho(W) +\Z$ (see also \cite[Lemma 4.51]{Y1}), one can prove that $\chi_W: D \to \mathbb{Q}/\Z$ is a group homomorphism
and we obtain a character $\widehat{\chi}_W \in \irr(D)$ defined 
by $$\widehat{\chi}_W(\alpha)= \exp\left(2\pi \sqrt{-1} \chi_W(\alpha)\right).$$  

Define a linear map $\tau_W: V_D\to V_D$ by  $$\tau_W(x)= \widehat{\chi}_W(\alpha) x, \quad x\in V^\alpha.$$ 
Then $\tau_W\in \Aut(V_D)$. 
Note that $\tau_W$ acts trivially on $V^0$ and $\tau_W =\mathrm{Id}_{V_D}$ if and only if 
$\chi_W(\alpha)=0$ for all $\alpha \in D$.

Let $A_W$ be a subgroup of $D_W$ such that $\widehat{A_W}$ is a maximal abelian subgroup of $\widehat{D_W}$.  
Let $\chi$ be a linear character of $\widehat{A_W}$,  and let $F_\chi$ be the corresponding irreducible $\widehat{A_W}$-module.  Then $W_\chi=F_\chi\otimes W$ is an irreducible $V_{A_W}$-module. Define the induced module
\[
X_{W_\chi}:= V_D\boxtimes_{V_{A_W}}W_\chi.
\]
As a $V^0$-module, one has 
$$X_{W_\chi}\cong \bigoplus_{\alpha \in D/A_W} V^{\alpha} \boxtimes_{V^0} W_\chi.$$
Then $X_{W_\chi}$ is an irreducible $\tau_W$-twisted $V_D$-module.  As a $V_{D_W}$-module,  $X_{W_\chi}$ splits as 
\[ X_{W_\chi}\cong
\bigoplus_{s \in D/D_W}  \left(\mathrm{Ind}_{\widehat{A}_W}^{\widehat{D}_W} e^s\otimes  F_\chi\right) \otimes W^s, 
\]
where $e^s$ denotes the one-dimensional representation corresponding to the coset $s\in D/D_W.$


\section{Simple current extensions and permutation orbifolds} \label{S:4}

In this section, we study the interplay between simple current extensions and permutation orbifolds. In particular, we construct and analyze twisted modules for extensions of $V^{\otimes k}$ arising from permutation symmetries, and develop an induced-module approach in this setting. 

Throughout this section, let $V$ be a rational, $C_2$-cofinite and self-dual VOA of CFT type, and let $k\geq 2$ be a positive integer. 
Let $U$ be a simple current extension of the tensor product VOA $V^{\otimes k}$. Then there exists a finite abelian group $H\le  \Aut(U)$ such that $U^H=V^{\otimes k}.$ 

Let $D=\irr(H)$ be the set of all irreducible characters of $H$. For $\alpha \in D$, define  
$$U_\alpha:= \{ u\in U\mid h(u)= \alpha(h)u\   \text{for}\ \text{all}\  h\in H\}.$$ Then
$$U= \oplus_{\alpha \in D} U_\alpha\ \text{and}\  U_{\mathrm{Id}}= V^{\otimes k}.$$    
Note that $D$  also forms a finite abelian group.  

\begin{rem}In the notation of Section~\ref{S:3.4}, $U$ is a $D$-graded simple current extension of
$V^{\otimes k}$: we identify $V^0=V^{\otimes k}$ and $V^\alpha=U_\alpha$.
Thus, $U=V_D$ and $D^\ast=\irr(D)\cong H$.
In this section, we use the notation $U=\bigoplus_{\alpha\in D}U_\alpha$ throughout.
\end{rem}

\subsection{Permutation symmetries and lifted actions}

In this subsection, we study permutation symmetries of the simple current extension $U$ and their lifts. This provides the framework for constructing twisted modules later.
Recall that the symmetric group $S_k$ of degree $k$ acts on $V^{\otimes k}$ and the set of irreducible $V^{\otimes k}$-modules naturally. Define $$A_D:=\{ \tau\in S_k\mid \{U_\alpha\circ \tau\mid \alpha\in D\} = \{U_\alpha \mid \alpha\in D\} \}.$$ By Proposition \ref{lift} (see also \cite[Theorem 3.3]{Sh04}), every $\tau\in A_D$ admits a lift
$\tilde{\tau}\in \Aut(U)$, unique up to multiplication by an element of $D^\ast\cong H$.
Hence, we obtain a subgroup of $\Aut(U)$ isomorphic to the semidirect product
$$
H \rtimes A_D
$$
where the action of $A_D$ on $H$ is induced by permuting the homogeneous summands
$\{U_\alpha\}_{\alpha\in D}$.
Now assume that the $k$-cycle $\sigma=(1\,2\,\cdots\,k)$ belongs to $A_D$.
Set 
$$G= H\rtimes \langle \sigma\rangle. $$  
Then $G$ is a solvable subgroup of $\Aut(U)$. Since $U^H=V^{\otimes k}$, we have 
$$U^G =(U^H)^{\langle\sigma\rangle}= (V^{\otimes k})^{\langle\sigma\rangle}.$$ As a consequence, every irreducible $U^G$-module can be realized as a submodule of an irreducible $\sigma^i$-twisted $V^{\otimes k}$-module for some  $i\in\{0,1, \dots,k-1\}$. 

The \textbf{goal} of this paper is to construct irreducible $g$-twisted modules for $g\in G$. In particular, we will generalize the induction theory for simple current extensions to the setting of permutation-twisted modules, and develop an induced-module construction for $\sigma$-twisted $U$-modules.    

Recall that irreducible $\sigma$-twisted $V^{\otimes k}$-modules are given by 
\[
\{T_\sigma(W)\mid W\in \irr(V)\}, 
\]
where $T_{\sigma}(W)$ is defined as in Section \ref{sec3}. Note also that the irreducible modules for $V^{\otimes k}$ are of the form $$M_1\otimes \cdots\otimes M_k,$$ where $M_1,\dots, M_k$ are irreducible $V$-modules.  
Therefore, for each $\alpha\in D$, the corresponding component $U_\alpha$ is of the form $$U_\alpha= M_1^\alpha \otimes \cdots\otimes M_k^\alpha$$ for some irreducible $V$-modules $M_1^\alpha, \dots,  M_k^\alpha$.

As in Section \ref{S:3.4}, we introduce a notion of induced modules for $\sigma$-twisted modules.  For any $W\in \irr(V)$, define \emph{the stabilizer subgroup} $$D_W= \{ \alpha\in D\mid U_\alpha \boxtimes_{V^{\otimes k}} T_\sigma(W) \cong T_{\sigma}(W)\}.$$ 
By Theorem \ref{FP}, $\alpha \in D_W$ if and only if  
$$M_1^\alpha \boxtimes_{V} \cdots\boxtimes_V M_k^\alpha\boxtimes_V W\cong W.$$ 
In particular,  $D_W$ is a subgroup of $D$. By the associativity and commutativity of the fusion products,   $D_W$ is stable under the action of $\sigma$. 

Let $SC(V)$ be the set of all (inequivalent) simple current modules of $V$ and set $$\overline{U} = \{ M_1^\alpha \boxtimes_{V} \cdots\boxtimes_V M_k^\alpha \in SC(V)\mid \alpha \in D\}.$$ Define
$$ \mathcal{C}_W= \{N\in SC(V)\mid N\boxtimes_V W\cong W\},\quad \overline{C}_W =\overline{U}\cap  \mathcal{C}_W.$$  Then $$D_W=\{\alpha\in D\mid M_1^\alpha \boxtimes_{V} \cdots\boxtimes_V M_k^\alpha \in \overline{C}_W\}.$$

As in Section \ref{S:3.4}, set $S_W= D/D_W$. For any $s=\alpha+D_W \in S_W$,  define $$W^s= M_1^\alpha \boxtimes_{V} \cdots\boxtimes_V M_k^\alpha\boxtimes_V W.$$  Then  $$T_\sigma(W^s) \cong U_\alpha \boxtimes_{V^{\otimes k}} T_\sigma(W).$$   By the associativity  and commutativity of the fusion products, the isomorphism class of $W^s$ is independent of the choice of the representative $\alpha$ in $s$ and  $D_W=D_{W^s}$. In particular, $W^s\cong W$ if and only if $s= D_W$.

Let  $s\in S_W$. For each $\alpha\in D_W=D_{W^s}$, fix a $V^{\otimes k}$-intertwining operator $$I^\alpha_s(\cdot, z),  \text{of}\  \text{type} \ \binom{T_\sigma(W^s)}{U_\alpha \quad  T_\sigma(W^s)}.$$ We choose  $I^0_s(\cdot, z)$ to be the vertex operator on the $\sigma$-twisted $V^{\otimes k}$-module $T_\sigma(W^s)$.   Then there is a $2$-cocycle $\lambda^t_s(\cdot , \cdot)$  such that 
\[
\langle \nu, I^\alpha_s(x^\alpha, z_1)I^\beta_s(x^\beta,z_2)w\rangle
= \lambda^t_s(\alpha, \beta) \langle \nu, I_s^{\alpha+\beta}( Y_{V_D}(x^\alpha, z_0)x^\beta,z_2)w\rangle|_{z_0=z_1-z_2}, 
\] 
where $x^\alpha\in U_\alpha, x^\beta\in U_\beta, w\in T_\sigma(W^s)$ and $\nu \in (T_\sigma(W^s))^*$.

We use the superscript $t$ to emphasize that $\lambda_s^t$ is defined using intertwining operators for $\sigma$-twisted $V^{\otimes k}$-modules; this distinguishes it from the untwisted cocycles $\lambda_s$ appearing in the induction theory of Section \ref{S:3.4}.

Let 
\[
1 \to \C^* \to \widehat{D^t_W} \to D_W \to 1 
\]
be a central extension of $D_W$ with respect to the 2-cocycle $\lambda^t_s(\ ,\ )$.

Let $A^t_W$ be a subgroup of $D_W$ such that $\widehat{A^t_W}$ is a maximal abelian subgroup of $\widehat{D^t_W}$.  
For a linear character $\chi$ of $\widehat{A^t_W}$, let $F_\chi$ denote the corresponding irreducible $\widehat{A^t_W}$-module, and define 
$$T_\sigma(W)_\chi := F_\chi\otimes T_\sigma(W),$$ viewed as a twisted module over the subextension
$$
U_{A_W^t}:=\bigoplus_{\alpha\in A_W^t}U_\alpha \ \subseteq\ U.$$ Define the induced module
\[
X(T_\sigma(W)_\chi) := U\boxtimes_{U_{A^t_W}}T_\sigma(W)_\chi.  
\]
Now we state the main result on the structure of induced modules. 
\begin{thm} \label{thm4.2}
Assume $\sigma=(1\,2\,\cdots\,k)\in A_D$ and keep the notation above.

\begin{enumerate}[label=\textnormal{(\arabic*)}]
    \item The induced module $X\left(T_\sigma(W)_\chi\right)$ is an irreducible $g$-twisted $U$-module
for some $g\in H\rtimes\langle\sigma\rangle$. 
    \item As a $U_{D_W}$-module, $X(T_\sigma(W)_\chi)$ admits a decomposition
\begin{equation}\label{eq:decompXT}
X(T_\sigma(W)_\chi)\ \cong\
\bigoplus_{s\in D/D_W}
\left(\text{Ind}_{\widehat{A_W^t}}^{\widehat{D_W^{\,t}}}\bigl(e^s\otimes F_\chi\bigr)\right)
\otimes T_\sigma(W^s),
\end{equation}
where $e^s$ denotes a one-dimensional $\widehat{A_W^t}$-module depending on the coset
$s\in D/D_W$ (transported from $s$ via the $D$-action), and
\[
\dim\left(\text{Ind}_{\widehat{A_W^t}}^{\widehat{D_W^{\,t}}}(e^s\otimes F_\chi)\right)
= |D_W/A_W^t|,
\]
which is independent of $s$.
\item  $\text{Ind}_{\widehat{A_W^t}}^{\widehat{D_W^{\,t}}}(e^s\otimes F_\chi)$ is irreducible
for each $s\in D/D_W$ by Mackey's irreducibility criterion.
\end{enumerate}
\end{thm}

\subsection{Classification of twisted modules}

Let $\sigma=(1\,2\,\cdots\,k)\in A_D$ be as above.  In this subsection, we study the structure and classification of $\sigma$-twisted modules for $U$, focusing on stabilizer subgroups and induced-module construction.

Let  $X$ be an irreducible $\sigma$-twisted $U$-module.  For $h\in H$, the module $X\circ h$ is an irreducible $h^{-1}\sigma h$-twisted $U$-module. In particular, $X\circ h\ncong X$  unless $h\in C_H(\sigma)= \{h\in H\mid  h^{-1}\sigma h=\sigma\}$.
Define  $$\phi:H\longrightarrow H,\qquad \phi(h)=[h,\sigma]=h\sigma h^{-1}\sigma^{-1}.$$ Since $\sigma$ normalizes $H$ and $H$ is abelian, $\phi$ is a group homomorphism.  Note that 
 $\ker \phi = \{h\in H\mid [h,\sigma]=1\}= C_H(\sigma)$.  

Now let $X$ be an irreducible $\sigma$-twisted module of $U$ that contains  $T_\sigma(W)$ as a $\sigma$-twisted $V^{\otimes k}$-module. Define the stabilizer of $X$ under the action of $H$ by 
$$S_X=\{ g\in H\mid X\circ g\cong X \}. $$ Then $S_X$ is a subgroup of the centralizer $ C_H(\sigma)=\{h\in H\mid h^{-1}\sigma h=\sigma\}.$ 
Let $D'$ and $D''$ be subcodes of $D$  such that $$U_{D'}=  U^{C_H(\sigma)} \  \text{and} \ U_{D''}= U^{S_X}.$$  Since $S_X\subset C_H(\sigma),$ it follows that $D'\le D''$. 

\begin{lemma}\label{D'}
We have $D'=(1-\sigma)D$.
\end{lemma}
\begin{proof} Recall that $D'$ is defined by $U_{D'}=  U^{C_H(\sigma)}$. Let $g\in C_H(\sigma)$. Then $g|_{U_{\alpha}} =g|_{U_{\sigma\alpha}}$ for any $\alpha\in D$. It implies $g|_{U_{(1-\sigma)\alpha}} = 1$. Therefore, 
$(1-\sigma)\alpha \in D'$ and $D'\supset (1-\sigma)D$.  On the other hand, for any coset $\alpha +(1-\sigma)D \in D/ (1-\sigma)D$, we have $$\sigma \alpha+(1-\sigma)D =\alpha +(1-\sigma)D. $$ Therefore, 
every character of $D/ (1-\sigma)D$ lifts to an element in $C_H(\sigma)$. Hence  
$D' \subset (1-\sigma) D$. Therefore  $D'=(1-\sigma)D$.    
\end{proof}

\begin{lemma} \label{TSX}
Let $N$ be an irreducible $\sigma$-twisted $U_{D''}$-submodule of $X$ that contains  $T_\sigma(W)$. Then  $N\cong T_\sigma(W)$ as $\sigma$-twisted $V^{\otimes k}$-modules.   
\end{lemma}
    
\begin{proof}
Since $S_X$ stabilizes $X$, it acts projectively on $X$ with a $2$-cocycle $\alpha_X$. 
By \cite{MT} (see also \cite{DRX}), $X$ admits a decomposition
 \[
 X= \bigoplus_{\lambda\in \Lambda_{S_X,\alpha_X}} W_\lambda\otimes X_\lambda,
 \]
 where $\Lambda_{S_X,\alpha_X}$ denotes the set of all irreducible projective characters of $S_X$ with the $2$-cocycle $\alpha_X$, $W_\lambda$ is the corresponding projective $S_X$-module and each $X_\lambda$ is an irreducible $\sigma$-twisted $U_{D''}$-module.   

By the definition of $S_X$, we have  $X\circ h \ncong X$ for any $h\in H\setminus S_X$. It follows that $X_\lambda\circ h \ncong X_\lambda$ for all  $\lambda\in \Lambda_{S_X,\alpha_X}$  and $h\in H\setminus S_X$ 
(see \cite[Theorem 2]{MT} or \cite[Theorem 3.2]{DRX}). This implies that $X_\lambda$ is still irreducible as a twisted module for $V^{\otimes k}= U_{D''}^{H/S_X}$.  Then $X_\lambda \cong  T_\sigma(W)$ as desired.  
\end{proof}   
 \medskip

By Lemma \ref{TSX},   we have $D'\le D''\le D_W$ where $$D_W=\{\alpha \in D \mid U_\alpha \boxtimes_{V^{\otimes k}} T_\sigma(W) \cong T_{\sigma}(W)\} $$ for $W\in \irr(V)$.  Now consider the submodule  
\[
U_{D_W} \boxtimes_{V^{\otimes k}} T_{\sigma}(W) =\oplus_{\alpha\in D_W} U_\alpha \boxtimes_{V^{\otimes k}}  T_{\sigma}(W) = |D_W|\cdot  T_{\sigma}(W). 
\] 
It can be viewed as an object in $\mathrm{Rep}(U_{D_W})$.   Therefore, 
it is a sum of irreducible $g_i$-twisted  modules for some $g_i$'s in  $G=H\rtimes\langle \sigma\rangle$. In general, it may involve more than one $g_i$,  and it may not contain a $\sigma$-twisted $U_{D_W}$-module.

\begin{lemma} \label{lem4.5}
Every irreducible $\sigma$-twisted $U_{D_W}$-module containing $T_{\sigma}(W)$ as a submodule can be realized in $U_{D_W} \boxtimes_{V^{\otimes k}} T_{\sigma}(W)$.
\end{lemma}

\begin{proof}
Let $X$ be an irreducible $\sigma$-twisted module of $U_{D_W}$ that contains  $T_\sigma(W)$.  
Then $X\cong \C^n \otimes T_\sigma(W)$ as a $\sigma$-twisted $V^{\otimes k}$-module and we have 
\[
\dim\left( \mathrm{Hom}_{U_{D_W}}\left( U_{D_W} \boxtimes_{V^{\otimes k}} T_{\sigma}(W),X\right)\right)  =
\dim \left(\mathrm{Hom}_{V^{\otimes k}} \left(T_{\sigma}(W), X \right)\right)=n \neq 0.  
\] 
Hence,  $U_{D_W} \boxtimes_{V^{\otimes k}} T_{\sigma}(W)$ contains a submodule isomorphic to $X$.  In fact, $U_{D_W} \boxtimes_{V^{\otimes k}} T_{\sigma}(W)$ contains exactly $n$-copies of $X$. 
\end{proof}

\begin{lemma} \label{UD'}
   Let $T$ be an irreducible $\sigma$-twisted module of $U_{D'}$ containing  $T_{\sigma}(W)$. Then $T\cong T_\sigma(W)$ as a $\sigma$-twisted $V^{\otimes k}$-module. Moreover, 
   \[
   U_{D'} \boxtimes_{V^{\otimes k}} T_{\sigma}(W) \cong  \bigoplus_{h\in H/C_H(\sigma)} T\circ h
   \]
   as an object in  $\mathrm{Rep}(U_{D'})$. 
\end{lemma}

\begin{proof}
   Since $D'\le D''$, $T\cong T_\sigma(W)$ follows from Lemma \ref{TSX}.  
For any $h\in H/C_H(\sigma)$,  $T\circ h $ is an irreducible $h^{-1}\sigma h$-twisted module of $U_{D'}$; in particular, it is an object in $\mathrm{Rep}(U_{D'})$. Then, by Theorem \ref{KO} (2), we have
\[
\dim\left( \mathrm{Hom}_{U_{D'}}\left( U_{D'} \boxtimes_{V^{\otimes k}} T_{\sigma}(W), T\circ h\right)\right)  =
\dim \left(\mathrm{Hom}_{V^{\otimes k}} (T_{\sigma}(W), T)\right)=1.  
\] 
Therefore, for all $h\in H/C_H(\sigma)$, $T\circ h$ can be viewed as a submodule of $U_{D'} \boxtimes_{V^{\otimes k}} T_{\sigma}(W)$ with multiplicity $1$ and we have 
\[
 U_{D'} \boxtimes_{V^{\otimes k}} T_{\sigma}(W) \supset \bigoplus_{h\in H/C_H(\sigma)} T\circ h.
\]
Since $|D'| = |H/ C_H(\sigma)|$, the equality holds and we have the desired result.    
\end{proof}

Note also that  $T$ is the unique irreducible $\sigma$-twisted module in  $U_{D'} \boxtimes_{V^{\otimes k}} T_{\sigma}(W)$
since  $h^{-1}\sigma h =\sigma$ implies $h\in C_H(\sigma)$.

\begin{lemma}
Let $T\subset X$ be an irreducible $\sigma$-twisted submodule of $U_{D''}$. Then for any $1\neq g\in C_H(\sigma)/S_X$, we have $T\circ g \ncong T$ as $\sigma$-twisted $U_{D''}$-modules.
\end{lemma}

\begin{proof}
Suppose otherwise. Then  $S= \{g\mid  T\circ g\cong T\}$ is a nontrivial subgroup. Then $S$ acts projectively on $T$, and $T$ will split into a sum of more than one irreducible twisted $U_{D'}$-modules. By Lemma \ref{TSX}, we have $T\cong T_\sigma(W)$ as a twisted $V^{\otimes k}$-module, which is impossible.  
\end{proof}

We next describe the multiplicities and number of irreducible $\sigma$-twisted $U_{D_W}$-modules  that contain $T_{\sigma}(W)$ as a submodule.
 
\begin{thm} \label{thm:perm_induction}

Let $W$ be an irreducible $V$-module, and set  $d=|D_W/D'|$.
Let $j$ be the number of inequivalent irreducible $\sigma$-twisted $U_{D_W}$-modules  that contain $T_{\sigma}(W)$ as a submodule. If $j\neq 0$, then 
\begin{enumerate}[label=\textnormal{(\arabic*)}]
\item The ratio $d/j$ is a square of an integer.

\item Every irreducible $\sigma$-twisted $U_{D_W}$-module is isomorphic to $\C^n\otimes T_\sigma(W)$ as twisted module of $V^{\otimes k}$, where $n=\sqrt{d/j}$.
\item  Moreover, $j = |D''/D'| = |C_H(\sigma)/S_X|$; that is, there are exactly $|D''/D'|= |C_H(\sigma)/S_X|$ inequivalent irreducible $\sigma$-twisted $U_{D_W}$-modules containing $T_{\sigma}(W)$ as a submodule.
\end{enumerate}
\end{thm}

\begin{proof}
Let $T$ be an irreducible $\sigma$-twisted $U_{D'}$-module such that
$
T\cong T_\sigma(W)$
as $\sigma$-twisted $V^{\otimes k}$-modules.  
By Lemma \ref{UD'} and the definition of $D_W$, we have  
\[
U_{D_W}\boxtimes_{U_D'} T \cong d\cdot T_\sigma(W)  
\]
as  a $\sigma$-twisted $V^{\otimes k}$-module.
Since  $$U_{D_W} \boxtimes_{V^{\otimes k}} T_{\sigma}(W)\cong U_{D_W} \boxtimes_{U_D'} \left(U_{D'} \boxtimes_{V^{\otimes k}} T_{\sigma}(W)\right),$$ the module 
$U_{D_W}\boxtimes_{U_D'} T$ can be viewed as a submodule of  $U_{D_W} \boxtimes_{V^{\otimes k}} T_{\sigma}(W)$. 
Recall that $T$ is the unique irreducible $\sigma$-twisted $U_{D'}$-module occurring in $U_{D'} \boxtimes_{V^{\otimes k}} T_{\sigma}(W)$ with multiplicity $1$. 

Statements (1) and (2) now follow from Lemma \ref{lem4.5} and its proof,  and  (3) follows from Lemmas \ref{TSX} and \ref{lem4.5}.
\end{proof}

\section{Explicit examples}
In this section, we illustrate the general theory developed in the previous sections by several families of examples. We first consider tensor products of VOSAs and subVOAs associated with binary codes, which provide a natural source of permutation-stable simple current extensions. We then discuss an example coming from the Moonshine VOA and a $3C$-element of the Monster, and finally present an example in which the multiplicity $n$ in the induced-module construction is greater than $1$. 

\subsection{Tensor products of VOSAs and code subVOAs } \label{VOSA}

We begin with examples that arise from the tensor products of VOSAs and their code subVOAs. After recalling the basic construction, we apply the results of Sections 3 and 4 to describe irreducible twisted modules for these code subVOAs, and then work out several concrete examples. 

Let  $V=V_{\overline{0}}\oplus V_{\overline{1}}$ be a $\frac{1}{2}\mathbb Z$-graded VOSA
such that $$V_{\overline{0}} =\oplus_{n\in \Z} V_n\ \text{and}\  V_{\overline{1}}= \oplus_{n\in \frac{1}2+\Z} V_n .$$  Let $\theta$ be the canonical involution of $V$, i.e., $\theta$ acts as $1$ on $V_{\overline{0}}$ and  $-1$ on $V_{\overline{1}}$.  

For a vector $\alpha=(\overline{\alpha_1}, \dots, \overline{\alpha_k})\in \Z_2^k$, define the homogeneous subspace  $$ V^\alpha := V_{\overline{\alpha_1}} \otimes \cdots \otimes V_{\overline{\alpha_k}}\subset V^{\otimes k}.$$ Then $$ V^{\otimes k}=\bigoplus_{\alpha\in \Z_2^k}V^\alpha=\bigoplus_{(\overline{\alpha_1}, \dots, \overline{\alpha_k})\in \Z_2^k}V_{\overline{\alpha_1}} \otimes \cdots \otimes V_{\overline{\alpha_k}}, $$ 
and each $V^\alpha$ is an irreducible module over $V^0= V_{\overline{0}}\otimes \cdots \otimes V_{\overline{0}}$. Let $\mathcal{C}\subset \Z_2^k$ be an even linear subcode, and define
\[
V_\mathcal{C}=\bigoplus_{\alpha\in \mathcal{C}} V^\alpha. 
\]   
Then $V_{\mathcal{C}}$ is a vertex operator subalgebra of the even part $(V^{\otimes k})_{\overline{0}}$. 
In particular, $$(V^{\otimes k})_{\overline{0}}= V_\mathcal{E},$$ where $\mathcal{E}$ is the code consisting of all even codewords in $\Z_2^k$.  


Let $\sigma$ be a $k$-cycle and let $\mathcal{C}$ be an even binary code such that $\sigma(\mathcal{C}) =\mathcal{C} $. 
Then  $\sigma\in \Aut(V^{\otimes k})$ preserves the subVOA $V_{\mathcal{C}}$. We also use $\sigma$ to denote the restriction $\sigma|_{V_{\mathcal{C}}}$.   


Let $W$ be an irreducible $V_{\overline{0}}$-module, and define $$W_1:= V_{\overline{1}}\boxtimes _{V_{\overline{0}}} W.$$ Let $$\mathcal{V}_O =  \{ W\in \irr(V_{\overline{0}})\mid 
q(W_1)- q(W) \in \Z\}$$ and 
$$\mathcal{V}_E =  \{ W\in \irr(V_{\overline{0}})\mid 
q(W_1) - q(W)\in \frac{1}{2}+\Z\},$$
where $q(W)$ (resp.\ $q(W_1)$) denotes the conformal weight of the 
$V_{\overline 0}$-module $W$ (resp.\ $W_1$), namely the lowest $L(0)$-eigenvalue 
appearing in $W$ (resp.\ $W_1$).

By the general theory of simple current extensions  \cite{Y1,La}, 
$$
V\boxtimes_{V_{\overline{0}}} W = W \oplus W_1.
$$
We distinguish the following three cases: 
\begin{itemize}
    \item 
If $W \in \mathcal{V}_E$, then $W \ncong W_1$ as $V_{\overline{0}}$-modules, and $W \oplus W_1$ is an irreducible $V$-module. 
\item 
If $W \in \mathcal{V}_O$ and $W \ncong W_1$, then $W \oplus W_1$ is an irreducible $\theta$-twisted $V$-module.
\item If $W \in \mathcal{V}_O$ and $W \cong W_1$, then $W$ admits two inequivalent $\theta$-twisted module structures.
\end{itemize}

Recall that $\theta_1=\theta\otimes 1\otimes \cdots \otimes 1$ denotes the automorphism of $V^{\otimes k}$ acting as $\theta$ on the first tensor factor.
By Theorem \ref{thmSVOSA},  we obtain the corresponding twisted modules for $V^{\otimes k}$ and for its even part $V_{\mathcal{E}}$. More precisely:
\begin{itemize}
\item If $W\in \mathcal{V}_E$, then 
$T_\sigma(W\oplus W_1)$ is an irreducible $\theta_1^{k-1}\sigma$-twisted $V^{\otimes k}$-module, and   $T_\sigma(W)$ and $T_\sigma(W_1)$ are irreducible $\theta_1^{k-1}\sigma$-twisted modules for $V_{\mathcal{E}} =(
V^{\otimes k})_{\overline{0}}$.  

\item If $W\in \mathcal{V}_O$ and $W\ncong W_1$, then       
$T_\sigma(W\oplus W_1)$ is an irreducible $\theta_1^{k}\sigma$-twisted $V^{\otimes k}$-module, and $T_\sigma(W)$, $T_\sigma(W_1)$ are irreducible $\theta_1^{k}\sigma$-twisted $V_{\mathcal{E}}$-modules. 

\item If $W\in \mathcal{V}_O$ and $W\cong W_1$, then       
$T_\sigma(W)$ is an irreducible $\theta_1^{k}\sigma$-twisted $V^{\otimes k}$-module with two inequivalent module structures, and it is  irreducible as a  $\theta_1^{k}\sigma$-twisted $V_{\mathcal{E}}$-module. 
\end{itemize}

We now apply this to the code subVOA $V_{\mathcal{C}}\subset V_{\mathcal{E}}$. By Theorem 
\ref{thm:3.1} and  the discussion above, every irreducible $\sigma$-twisted $V_{\mathcal{C}}$-module is isomorphic, as a $V_{\overline{0}}^{\otimes k}$-module,   to  $T_\sigma(W)$ for some irreducible  $V_{\overline{0}}$-module $W$. 
Moreover, if $k$ is odd (resp., even) and $W\in  \mathcal{V}_E$ (resp.,  $W\in  \mathcal{V}_O$), then $T_\sigma(W)$ can  be viewed as an irreducible $\sigma$-twisted $V_{\mathcal{C}}$-module.

To determine how many such twisted structures occur, set  $$H=\{g\in \Aut(V_\mathcal{C})\mid g|_{V^0}=\mathrm{id}_{V^0}\}.$$ Then $H\cong \irr(\mathcal C)\cong \mathcal C^\ast $. 
Recall  from Section \ref{S:4} the group homomorphism  $$\phi: H \to  H,\qquad \phi(h)= [h, \sigma]=h\sigma h^{-1}\sigma^{-1},$$  and  set   $$K=\mathrm{Im} \phi \cong H/C_H(\sigma).$$  
If $\theta_1\in K$, then there exists an automorphism $h\in H$ such that 
$h\sigma h^{-1} = \theta_1\sigma$.  Then $ T\circ h$  is an irreducible $\sigma$-twisted $V_\mathcal{C}$-module for any irreducible $\theta_1\sigma$-twisted $V_\mathcal{C}$-module $T$.  

Since $\mathcal{C}$ is even, $V_{\overline{\alpha_1}} \boxtimes \cdots \boxtimes V_{\overline{\alpha_k}} \cong V_{\overline{0}}$ for any $\alpha=(\overline{\alpha_1},\ldots,\overline{\alpha_k})\in \mathcal{C}$.
Hence for any $W\in \mathrm{Irr}(V_{\overline{0}}),$ $$\mathcal{C}_W = \{ \alpha\in \mathcal{C}\mid V_{\overline{\alpha_1}} \boxtimes \cdots \boxtimes V_{\overline{\alpha_k}}
\boxtimes W\cong W\} =\mathcal{C}.$$ 
Let $\mathcal{C}'$ be the subgroup of $\mathcal C$ defined by $V_{\mathcal{C}'}=(V_{\mathcal{C}})^{C_H(\sigma)}$. Then Lemma \ref{D'} gives $\mathcal{C}'= (1-\sigma)\mathcal{C}$.

We now distinguish two cases.

\noindent \textbf{Case 1:} $k$ is even. 

Let $W\in  \mathcal{V}_O$. Then $T_\sigma(W)$ can be viewed as an irreducible $\sigma$-twisted module for $V_{\mathcal{C}}$. In this case,    $\theta_1\notin  K$ and thus $\theta_1\sigma$ is  not conjugate to $\sigma$ by   $H$.  Therefore, by 
Theorem \ref{thm:perm_induction}, there are exactly
$|\mathcal{C}/\mathcal{C}'|$ inequivalent irreducible $\sigma$-twisted $V_{\mathcal{C}}$-module structures on $T_\sigma(W)$.

\noindent \textbf{Case 2:} $k$ is odd.

Let $W\in  \mathcal{V}_E$. Then  $T_\sigma(W)$ can be viewed as an irreducible $\theta_1\sigma$-twisted module for $V_\mathcal{C}$.
Since $\theta_1\in K=\mathrm{Im}\phi$  when $k$ is odd, there exists  $h\in H$ such that $h \sigma h^{-1}=  \theta_1\sigma.$ Hence $T_\sigma(W)\circ h$ is an irreducible $\sigma$-twisted module for $V_\mathcal{C}$.
Again by Theorem \ref{thm:perm_induction}, there are $|\mathcal{C}/\mathcal{C}'|$ module structures on $T_\sigma(W)$.  

 We now illustrate the general discussion above with binary code VOAs arising from the free-fermion VOSA  $$V=L(1/2,0)\oplus L(1/2,1/2).$$ For an even binary code  $C$, let  $V_{C}$ denote the corresponding binary code VOA studied by Miyamoto \cite{Mi96}. We consider the following examples.
\begin{exa}  Let $C=\mathcal{E}$ be the code consisting of all even codewords of $\Z_2^k$.  Then  the dual code $$\mathcal{E}^* =\{(0, \dots,0), (1, \dots,1)\}.$$  Note that $(1,\dots,1)\in \mathcal{E}$ if $k$ is even, while $(1,\dots,1)\notin \mathcal{E}$ if $k$ is odd. 

For even  $k$ , set $n=k/2$. Then $$\mathcal{E}'= (1-\sigma)\mathcal{E}_k \cong  \mathcal{E}_n\oplus \mathcal{E}_n,$$ which is an index $2$ subcode of $ \mathcal{E}_k$.  Moreover, $V_{\mathcal{E}}$ is isomorphic to the lattice VOA $V_{D_n}$. In this case, $$|C/C'|=2,\qquad \mathcal{V}_O=\{L(1/2,1/16)\}.$$ Hence, there are exactly 
two irreducible $\sigma$-twisted modules of $V_{\mathcal{E}}$, and both are isomorphic to $T_\sigma(L(1/2,1/16))$ as  a $\sigma$-twisted $L(1/2,0)^k$-module.

For  odd $k$, we have  $$\mathcal{E}'= (1-\sigma)\mathcal{E}_k =\mathcal{E}_k, \qquad |C/C'|=1.$$  In this case, 
$V_{\mathcal{E}}$ is isomorphic to the affine VOA $L_{B_{(k-1)/2}}(1,0)$  for $k\geq 5$. If $k=3$, then  $V_{\mathcal{E}}\cong L_{A_1}(2,0)$.  
In this case, $$\mathcal{V}_E= \{ L(1/2,0), L(1/2,1/2)\},\quad
\mathcal{V}_O= \{ L(1/2,1/16)\}.$$ 

Since $\theta_1\in K$  when $k$ is odd, there exists $h\in H$ such that $h^{-1}\theta_1\sigma h =\sigma$. 
Therefore, there are exactly three irreducible $\sigma$-twisted modules for $V_{\mathcal{E}}$, namely,  $$T_\sigma(L(1/2,0)), \quad T_\sigma(L(1/2,1/2)),\quad  T_\sigma(L(1/2,1/16)) \circ h.$$

\end{exa}
Note that this example has also been discussed in \cite{MS}.
\begin{exa} Let $C=\langle (11111111), (11001100), (10101010)\rangle$. Then $C$ is stabilized by the $8$-cycle $\sigma=(1\ 2\ 3\ 4\ 5\ 6\ 7\ 8)$. 

In this case, $$C'= (1-\sigma)C= \langle (11111111), (10101010) \rangle,\quad |C/C'|=2.$$
Therefore, there are exactly two inequivalent irreducible
$\sigma$-twisted $V_C$-modules containing
$
T_\sigma\left(L(1/2,1/16)\right)
$
as a $\sigma$-twisted $L(1/2,0)^{\otimes 8}$-submodule.     

\end{exa}
\begin{exa} Let $k=6$ and let $\sigma=(1\ 2\ 3\ 4\ 5\ 6)$ act on $V_{\overline 0}^{\otimes 6}$  by cyclically permuting the tensor factors.  
Consider \[
\mathcal{C}
 = \{ (0,0,0,0,0,0), (1,1,1,1,1,1) \}, \qquad |\mathcal{C}| = 2.
\]
Then the corresponding code subVOA $V_{\mathcal{C}}$ is a simple $\mathbb{Z}_2$-graded
simple current extension of $V_{\overline{0}}^{\otimes 6}$.

 Since $k=6$ is even, the relevant twisted modules come from 
$$    W \in \mathcal{V}_O = \bigl\{ L(1/2,1/16) \bigr\}.$$
 Hence the relevant underlying $\sigma$-twisted $V_{\overline{0}}^{\otimes 6}$-module is  $    T_\sigma\bigl(L(1/2,1/16)\bigr).
$
 Since both elements of $\mathcal C$ are fixed by $\sigma$, we have
$
\mathcal C'=(1-\sigma)\mathcal C=\{ (0,0,0,0,0,0) \}$, so $ |\mathcal{C}'|=1$.  Hence 
$    \bigl|\mathcal{C}/\mathcal{C}'\bigr|
      = \frac{|\mathcal{C}|}{|\mathcal{C}'|}
      = \frac{2}{1}
      = 2.
$
Therefore,  by Theorem \ref{thm:perm_induction}, the underlying $V_{\overline{0}}^{\otimes 6}$-module $T_\sigma\bigl(L(1/2,1/16)\bigr)$ admits exactly two inequivalent
 $\sigma$-twisted $V_{\mathcal{C}}$-module structures.

\begin{rem}

In this example, the VOA $V_{\mathcal C}$ may also be interpreted in terms of
permutation orbifolds of Ising models.
Indeed, since $V_{\overline{0}}=L(1/2,0)$ and $V_{\overline{1}}=L(1/2,1/2)$, one has
$$
V_{\mathcal C}
= L(1/2,0)^{\otimes 6}
\;\oplus\;
L(1/2,1/2)^{\otimes 6},
$$
which is a $\mathbb Z_2$-graded simple current extension of
$L(1/2,0)^{\otimes 6}$.
The $6$-cycle $\sigma=(1\ 2\ 3\ 4\ 5\ 6)$ acts naturally on $L(1/2,0)^{\otimes 6}$, and the
extension above is $\sigma$-stable.
Thus $V_{\mathcal C}$ may be regarded as a simple current extension of the
permutation orbifold $\left(L(1/2,0)^{\otimes 6}\right)^{\langle\sigma\rangle}$.

\end{rem}

\end{exa}

\subsection{Moonshine VOA, $V_{\sqrt{2}E_8}^+$ and $3C$-element of the Monster}
In this subsection, we apply the general construction to the Moonshine VOA and 
an element of the $3C$-class of the Monster.

We first review the realization of $V^\natural$ from three copies of $V_{\sqrt{2}E_8}^+$ and the explicit construction of an  automorphism  $g\in \Aut(V^\natural)$ in the conjugacy class $3C$ from \cite{Mi04,Shi}.

We use $V_{\sqrt{2}E_8}^+$ to denote the fixed point subVOA of the lattice VOA $V_{\sqrt{2}E_8}$ by an involution $\theta$, which is a  lift of the $-1$ isometry of the lattice. The representation theory for $V_{\sqrt{2}E_8}^+$ has been studied by many authors \cite{ADL,Mi04,Shi}. In particular,  all
irreducible modules and their fusion rules are determined.  We use $[M]$ to denote the isomorphism class of $M$ for any irreducible $V_{\sqrt{2}E_8}^+$-module $M$. 

\begin{prop}[\cite{ADL,Shi}]
The VOA $V_{\sqrt{2}E_8}^+$ has exactly $2^{10}$ inequivalent irreducible modules, and all of them are simple current modules.
Moreover,  the set of all inequivalent irreducible modules $\irr (V_{\sqrt{2}E_8}^+)$ forms a 10-dimensional non-degenerate quadratic
space of  plus type over $\mathbb{F}_2$ with the sum defined by fusion rules and the quadratic form given by the conformal weights (in
$\frac{1}2 \Z/\Z$).
\end{prop}

Let $\Phi$ and $\Psi$ be maximal totally isotropic subspaces of $\irr (V_{\sqrt{2}E_8}^+)$, that is, maximal subspaces on which the quadratic form vanishes, such that $\Phi\cap \Psi = 0$. Let $$ (V_{\sqrt{2}E_8}^+)^3=V_{\sqrt{2}E_8}^+ \otimes  V_{\sqrt{2}E_8}^+ \otimes V_{\sqrt{2}E_8}^+$$ denote the tensor product of three copies of $V_{\sqrt{2}E_8}^+$.  For any $\{i, j\}\subset \{1,2,3\}$, set
\[
\Phi_{(i,j)} = \left\{(a_1, a_2,a_3) \in \irr (V_{\sqrt{2}E_8}^+)^3\mid a_i = a_j \in \Phi,\text{ and } a_k = 0 \text{ if } k\notin \{i,j\}\right\}
\]
and
\[
\Psi_{(1,2,3)} = \left\{(b, b , b)\in  \irr (V_{\sqrt{2}E_8}^+)^3\mid b\in \Psi\right\}.
\]

\begin{thm}[\cite{Shi}]\label{trans}
Let $ \mathcal{S}(\Phi,\Psi) = \text{Span}_{\mathbb{F}_2} \{\Phi_{(1,2)}, \Phi_{(1,3)}, \Psi_{(1,2,3)} \}.$ Then  $\mathcal{S}(\Phi,\Psi)$ is a
maximal totally isotropic subspace of $\irr (V_{\sqrt{2}E_8}^+)^3$, and
\[
V^\natural \cong V(\mathcal{S}(\Phi,\Psi)) =\bigoplus_{[A\otimes B\otimes C] \in \mathcal{S}(\Phi,\Psi)} A\otimes B\otimes C.
\]
as modules for $ (V_{\sqrt{2}E_8}^+)^3$.
Moreover, the automorphism group of $V^\natural$ is transitive on the set of all full subVOAs isomorphic to $ (V_{\sqrt{2}E_8}^+)^{3}$.
\end{thm}

Now let $g$ be an automorphism of $ (V_{\sqrt{2}E_8}^+)^{3}$ given by the cyclic permutation of the $3$ tensor factors. 
Since $g$ preserves $\mathcal{S}(\Phi,\Psi)$, $g$ lifts to an automorphism of $V^\natural$.  

\begin{rem}
It was shown in \cite{Mi04} that the cyclic permutation of $ (V_{\sqrt{2}E_8}^+)^3$ lifts to an element of the conjugacy class $3C$ in $\Aut(V^\natural)$.   
\end{rem}

In this case, the grading group is  $$D= \mathcal{S}(\Phi,\Psi) \le\mathbb{F}_2^{30}.$$ 
Since all irreducible modules of $V_{\sqrt{2}E_8}^+$ are simple current modules, for any irreducible module $W\in \irr (V_{\sqrt{2}E_8}^+)$
we have 
\[
D_W= \{ [M_1\otimes M_2\otimes M_3]\in D\mid  M_1\boxtimes_{V_{\sqrt{2}E_8}^+} M_2\boxtimes_{V_{\sqrt{2}E_8}^+} M_3\cong V_{\sqrt{2}E_8}^+\}.
\]
Therefore, $$D_W= \text{Span}_{\mathbb{F}_2} \{\Phi_{(1,2)}, \Phi_{(1,3)} \}.$$ Note that 
$\Phi_{(2,3)}$ is also contained in $D_W$. 

By Lemma \ref{D'}, we also have  $$D'= (1-g)D = \text{Span}_{\mathbb{F}_2} \{\Phi_{(1,2)}, \Phi_{(1,3)} \}.$$ Hence, $D_W =D'$. Therefore, there is at most one irreducible 
$g$-twisted module of $U_{D_W}$ that contains $T_g(W)$, and it is isomorphic to $T_g(W)$
as a $g$-twisted module for $ (V_{\sqrt{2}E_8}^+)^{3}$.   
Since $D/D_W$ is naturally identified with $\Psi$ through
$b\mapsto (b,b,b)+D_W$, 
Theorem \ref{thm4.2} gives that the irreducible
$g$-twisted $V^\natural$-module has the form 
\[
\bigoplus_{[A]\in \Psi} T_g (A\boxtimes W)
\]
for some $W\in \irr (V_{\sqrt{2}E_8}^+)$.  

Since $g$ has  order $3$ and $\irr (V_{\sqrt{2}E_8}^+)\cong 2^{10}_+$,  $W$ must be orthogonal to $\Psi$; otherwise, there is an $A\in \Psi$ such that the difference between the conformal weights of $W$ and $A\boxtimes W$ is in $1/2+\Z$ . Then the weights of $T_g(W)$ and $T_g (A\boxtimes W)$ differ by $1/6 \mod 1/3 \Z$.  The module $\bigoplus_{[A]\in \Psi} T_g (A\boxtimes W)$  is  not a $\Z_3$-twisted module.

Thus,  $W\in \Psi$ since $\Psi$ is a maximal totally isotropic subspace of $\irr (V_{\sqrt{2}E_8}^+)$. Without loss, we may assume $W= V_{\sqrt{2}E_8}^+$. 
Then  the unique irreducible $g$-twisted module of $V^\natural$ 
has the form 
\[
\oplus_{[A]\in \Psi} T_g (A)= T_g(\oplus_{[A]\in \Psi} A). 
\]
 Since $\Psi$ is a maximal totally isotropic subspace of $\irr (V_{\sqrt{2}E_8}^+)$,  $X= \oplus_{[A]\in \Psi} A$ has a structure of holomorphic VOA of central charge $8$. Thus  $X\cong V_{E_8}$, the lattice VOA associated with the root lattice of type $E_8$ \cite{DM04}. We therefore obtain the following result.

\begin{thm}
The unique $3C$-twisted module of the Moonshine VOA $V^\natural$ is isomorphic to 
$T_g(V_{E_8})$ as a twisted module of $ (V_{\sqrt{2}E_8}^+)^{ 3}$. 
\end{thm}

\subsection{Some examples with multiplicity $n>1$}  
In this subsection, we give an example with multiplicity $n>1$ in the induced-module construction, illustrating the role of the quotient $D_W/D'$.

Let $L=\Z\alpha$ be a rank one lattice with $\langle \alpha, \alpha\rangle =4$. Then $L\cong \sqrt{2}A_1$ and the lattice VOA $V_L$ has 4 irreducible modules, up to equivalence, namely,  $$V_L, \quad V_{\frac{1}4\alpha+L},  \quad V_{\frac{1}2\alpha+L}, \quad 
V_{-\frac{1}4\alpha+L}.$$

Let $\mathcal{E}_8\subset\mathbb{Z}_2^8$ be the even binary code, and define 
$$2\mathcal{E}_8:=\{(2c_1,\dots,2c_8)\mid c\in\mathcal{E}_8\}\subset\mathbb{Z}_4^8,$$ viewed as a $\mathbb Z_4$-subcode of $\Z_4^{8}$. 
Let $D \subset \Z_4^{8}$ be a $\Z_4$-code generated by 
$2\mathcal{E}_8$  and the all-one vector $(1,1,1,1,1,1,1,1)$.  
Then  $D$ is  a self-dual $\Z_4$-code and 
\[
U:= \bigoplus_{(d_1, \dots,d_8)\in D} 
V_{\frac{1}4d_1\alpha + L} \otimes \cdots \otimes V_{\frac{1}4d_8 \alpha+L}
\] 
is  a holomorphic VOA of central charge $8$; hence,  $U\cong V_{E_8}$. 

Let $\sigma=(1\,2 \cdots 8)$ be an $8$-cycle, which acts on $V_L^{\otimes 8}$ as a cyclic permutation of the tensor factors. Note that $\sigma(D)=D$.
Hence, $\sigma$ lifts to an automorphism of $U$. 

The code $D$ contains a subcode $C=2\mathcal{E}_8$ and the  corresponding code subVOA is 
\[
U_C=\bigoplus_{(c_1, \dots,c_8)\in \mathcal{E}_8} 
V_{\frac{1}2c_1\alpha + L} \otimes \cdots \otimes V_{\frac{1}2c_8 \alpha+L}  \cong V_{D_8}.
\]
Moreover,  $$X= V_L\oplus V_{\frac{1}2\alpha+L}\cong V_\Z$$ is a $\frac{1}{2}\mathbb Z$-graded VOSA with $$X_{\bar{0}}= V_L, \qquad X_{\bar{1}} = V_{\frac{1}2\alpha+L}.$$  
Using the  notation of  Section \ref{VOSA}, we have 
\[
\mathcal{V}_O =\{V_{\frac{1}4\alpha+L},  V_{-\frac{1}4\alpha+L}\}.
\] 
Thus  $T_\sigma(V_{\frac{1}4\alpha+L})$ and $T_\sigma(V_{-\frac{1}4\alpha+L})$ 
can be viewed as irreducible $\sigma$-twisted $V_L^{\otimes 8}$-modules.   Since
$$C'=(1-\sigma)C= 2(\mathcal{E}_4\oplus \mathcal{E}_4)$$ and hence $|C/C'|=2$,  Theorem \ref{thm:perm_induction} implies that each of these modules admits exactly two inequivalent
$\sigma$-twisted $U_C$-module structures.

Now consider the full extension $U$. For  any $W\in \irr(V_{L})$, we have $D_W=D$. Moreover,  $$D'=(1-\sigma) D= 2(\mathcal{E}_4\oplus \mathcal{E}_4),$$ and hence  $d=|D/D'|=4$. Since $U$ is holomorphic, it has a unique $\sigma$-twisted module. Hence,  in the notation of Theorem \ref{thm:perm_induction}, we have $j=1$. It follows that the multiplicity of $T_\sigma(W)$ in the unique
irreducible $\sigma$-twisted $U$-module is $$n=\sqrt{d/j}=\sqrt{4}=2.$$ 
\section*{Acknowledgements}
The first author was  supported by the National Science and Technology Council of Taiwan (Grant No. NSTC 113-2115-M-001-012-MY3).  
This work was partially carried out during the second author's visit to the
Institut des Hautes \'Etudes Scientifiques (IHES), whose support, hospitality,
and stimulating research environment are gratefully acknowledged. The second
author was partially supported by the Natural Science Foundation of Xiamen
Municipality (Grant No. 3502Z202473007), the Natural Science Foundation of
Fujian Province (Grant No. 2024J01027), and the National Natural Science
Foundation of China (Grant Nos. 12571031 and 12131018).

\noindent {\bf Ching Hung Lam}: Institute of Mathematics, Academia Sinica, Taipei 10617, Taiwan; chlam@math.sinica.edu.tw

\noindent {\bf Nina Yu}: School of Mathematical Sciences, Xiamen University, Fujian, 361005, China; ninayu@xmu.edu.cn
\end{document}